\documentclass{amsart}
\usepackage{fullpage}
\usepackage{hyperref}
\usepackage[all]{xy}

\newtheorem{theorem}{Theorem}[section]
\newtheorem{lemma}[theorem]{Lemma}

\newtheorem{algorithm}[theorem]{Algorithm}
\theoremstyle{definition}
\newtheorem{defn}[theorem]{Definition}
\newtheorem{remark}[theorem]{Remark}

\numberwithin{equation}{section}

\newcommand{\FF}{\mathbb{F}}

\newcommand{\PP}{\mathbf{P}}

\newcommand{\calM}{\mathcal{M}}
\newcommand{\calO}{\mathcal{O}}

\DeclareMathOperator{\Aut}{Aut}
\DeclareMathOperator{\End}{End}
\DeclareMathOperator{\Ext}{Ext}
\DeclareMathOperator{\GL}{GL}
\DeclareMathOperator{\Gr}{Gr}
\DeclareMathOperator{\Hom}{Hom}
\DeclareMathOperator{\OG}{OG}

\DeclareMathOperator{\Proj}{Proj}

\DeclareMathOperator{\SO}{SO}
\DeclareMathOperator{\SpG}{SpG}

\DeclareMathOperator{\Sym}{Sym}

\newcommand{\avlink}[1]{\href{http://www.lmfdb.org/Variety/Abelian/Fq/#1}{\textsf{#1}}}

\newcommand{\Gap}{\textsc{GAP}}
\newcommand{\Magma}{\textsc{Magma}}
\newcommand{\SageMath}{\textsc{SageMath}}
\newcommand{\Singular}{\textsc{Singular}}

\newcommand{\arXiv}[3]{\href{https://arxiv.org/abs/#1}{arXiv:#1v#2} (#3)}

\begin{document}

\title{The relative class number one problem for function fields, III}
\author{Kiran S. Kedlaya}
\thanks{Thanks to Samir Canning, John Cremona, Xander Faber, Yongyuan Huang, Joan-Carles Lario, Jun Bo Lau, Bjorn Poonen, and David Roe for helpful discussions. The author was supported by NSF (grants DMS-1802161, DMS-2053473) and UC San Diego (Warschawski Professorship).}

\date{28 Dec 2023}

\begin{abstract}
We complete the solution of the relative class number one problem for function fields of curves over finite fields. Using work from two earlier papers,
this reduces to finding all function fields of genus 6 or 7 over $\FF_2$ with one of 40 prescribed Weil polynomials; one may then verify directly that three of these fields admit an everywhere unramified quadratic extension with trivial relative class group.
The search is carried out by carefully enumerating curves based on the Brill--Noether stratification of the moduli spaces of curves in these genera, and particularly Mukai's descriptions of the open strata.
\end{abstract}

\maketitle

\section{Introduction}

This paper continues and concludes the work done in \cite{part1, part2} on the \emph{relative class number one problem} for function fields of curves over finite fields (hereafter simply ``function fields''), building upon work of Leitzel--Madan \cite{leitzel-madan} and Leitzel--Madan--Queen \cite{leitzel-madan-queen}.
That is, we seek to identify finite extensions $F'/F$ of function fields for which the two class numbers are equal. 

To state the main result, we recall some context from the introduction of \cite{part1}. 
Given a finite extension $F'/F$ of function fields, we write
$C, C'$ for the curves corresponding to $F,F'$;
$q_F, q_{F'}$ for the orders of the base fields of $C,C'$;
$g_F, g_{F'}$ for the genera of $C, C'$;
and $h_F, h_{F'}$ for the class numbers of $F, F'$. 
Since the relative class number $h_{F'/F} = h_{F'}/h_F$ is an integer
(it is the order of the Prym variety of the covering $C' \to C$), the relative class number one problem reduces to the cases where $g_{F'} = g_F$ (a \emph{constant extension}) and where $q_F = q_{F'}$ (a \emph{purely geometric extension}).
Excluding the trivial cases of a constant extension of genus-0 function field and an extension with $F' \cong F$, one has the following result (see \cite{part1} for the tables).

\begin{theorem}[Solution of the relative class number one problem] \label{T:classification}
Let $F'/F$ be an extension of function fields of degree $d>1$ of relative class number $1$.
\begin{enumerate}
\item[(a)]
If $F'/F$ is constant and $g_F > 0$, then $q_F, d, g_F$, and the isogeny class of $J(C)$ appear in
\cite[Theorem~1.1]{part1}. In particular,
\[
(q_F, d, g_F) \in \{(2,2,1), (2,2,2), (2,2,3), (2,3,1), (3,2,1), (4,2,1)\}.
\]
\item[(b)]
If $F'/F$ is purely geometric, $g_F \leq 1$, and $g_{F'} > g_F$, then $q_F, g_F, g_{F'}$, and the isogeny classes of $J(C)$ and $J(C')$ appear in \cite[Table~3]{part1}. In particular,
\begin{gather*}
(q_F, g_F, g_{F'}) \in \{(2,0,1\mbox{--}4), 
(2,1,2\mbox{--}6), (3,0,1), (3,1,2), (3,1,3), (4,0,1), (4,1,2)\}.
\end{gather*}
\item[(c)]
If $F'/F$ is purely geometric, $g_F > 1$, and $q_F > 2$, then $d, g_F, g_{F'}, F$ appear in \cite[Table~4]{part1}. In particular,
\begin{gather*}
(q_F, d, g_F, g_{F'}) \in \{(3,2,2,3), (3,2,2,4), (3,2,3,5), (3,3,2,4), (4,2,2,3), (4,3,2,4)\}.
\end{gather*}
\item[(d)]
If $F'/F$ is purely geometric, $g_F > 1$, $q_F = 2$, and $d > 2$, then $d, g_F, g_{F'}, F$ appear in \cite[Table~5]{part1}. In particular,
\begin{gather*}
(d, g_F, g_{F'}) \in \{(3,2,4), (3,2,6), (3,3,7), (3,4,10), (4,2,5), (5,2,6), (7,2,8)\}.
\end{gather*}
\item[(e)] 
If $F'/F$ is purely geometric, $g_F > 1$, $q_F = 2$, and $d = 2$, then $g_F, g_{F'}, F$ appear in \cite[Table~6]{part1}. In particular,
\begin{gather*}
(g_F, g_{F'}) \in \{(2,3), (2,4), (2,5), (3,5), (3,6), (4,7), (4,8), (5,9), (6,11), (7,13)\}.
\end{gather*}
\item[(f)]
If $F'/F$ is neither constant nor purely geometric and $g_{F'} > g_F$, then $q_F = 2, q_{F'} = 4$, and $(g_F, g_{F'}, J(C), J(C'))$ is one of $(0,1,0, \avlink{1.4.ae})$ or $(1,2,\avlink{1.2.c}, \avlink{2.4.ae\_i})$ (using LMFDB labels to represent isogeny classes of abelian varieties).
\end{enumerate}
\end{theorem}
This statement is covered by \cite[Theorems~1.1, 1.2, 1.3]{part1} except for the following points.
\begin{itemize}
\item
Part (b) requires classifying curves of genus 6 over $\FF_2$ with one particular Weil polynomial. It is shown in \cite[Lemma~10.2]{part2} that there is a unique such curve.
\item
Part (d) requires showing that when $q_F = 2$, the extension $F'/F$ is cyclic. This is done in
\cite[Theorem~1.1]{part2} using constraints on the Weil polynomials found in \cite{part1}.
\item
Part (e) requires finding all curves of genus 6 and 7 over $\FF_2$ with Weil polynomials in a specific list of 40 entries found in \cite{part1} (see Table~\ref{table:genus 567 counts}). This brings us to the main result of the present paper, stated as
Theorem~\ref{T:main} below.
\end{itemize}

\begin{theorem} \label{T:main}
The following statements hold.
\begin{enumerate}
\item[(a)]
There are two isomorphism classes of curves $C$ of genus $6$
 over $\FF_2$ admitting
an \'etale double covering $C' \to C$ such that $\#J(C')(\FF_2) = \#J(C)(\FF_2)$.
The curves $C$ are Brill--Noether general with automorphism groups $C_3$ and $C_5$.
\item[(b)]
There is a unique isomorphism class of curves $C$ of genus $7$
 over $\FF_2$ admitting
an \'etale double covering $C' \to C$ such that $\#J(C')(\FF_2) = \#J(C)(\FF_2)$.
The curve $C$ is bielliptic with automorphism group $D_6$.
\end{enumerate}
\end{theorem}

As in \cite{part1}, given a candidate for $C$ it is straightforward to use \Magma{} to generate all of the \'etale double coverings $C' \to C$;
thus the main computational issue is to ``invert the Weil polynomial function'' on the output values indicated in Table~\ref{table:genus 567 counts}.
Unfortunately, the Weil polynomial function is in some sense a ``secure hash function'', in that its value generally does not reveal much useful information about the input. Examples where one can invert the function are often of some extremal nature, as in Lauter's approach to bounding the maximum number of points on a curve of fixed genus over a fixed finite field: one first enumerates the Weil polynomials consistent with a given point count, then attempts to rule in or out the various candidates. Much work has been done on the second step by Howe; see \cite{howe2021} for a recent survey of this problem.

Unfortunately, the techniques described in \cite{howe2021} do not seem to be applicable to the cases relevant to Theorem~\ref{T:main}. Fortunately, for genera up to 7 it is feasible to deploy a brute force strategy, i.e., to enumerate a collection of schemes known to include all curves with the given Weil polynomials and then filter through the results. One way to build such a collection is using singular plane curves; see \cite{faber-grantham} and \cite{faber-grantham-howe} for recent examples of this approach.

Here we take an alternate approach that accounts for the known geometry of moduli spaces of curves based on Petri's analysis of linear systems (see \S\ref{sec:canonical curves}). This amounts to a natural extension of the computation of the set of isomorphism classes of curves of genus $g$ over $\FF_2$ for $g=4$ by Xarles \cite{xarles}, based on the fact that a general canonical curve of genus 4 is a complete intersection of type $(2) \cap (3)$ in $\PP^3$;
and $g=5$ by Dragutinovi\'c \cite{dragutinovic}, based on the fact that a general canonical curve of genus 5 is a complete intersection of type $(2) \cap (2) \cap (2)$ in $\PP^4$.

While one cannot hope to give similar such descriptions in arbitary genus (see Remark~\ref{rmk:Brill-Noether}), they are available in genus 6 and 7 by work of Mukai \cite{mukai-cg, mukai}, although some care is required to use them over a nonclosed base field. As in \cite{faber-grantham} and \cite{faber-grantham-howe}, we short-circuit the searches using the Weil polynomial constraints, especially the number of $\FF_2$-rational points.
See Lemma~\ref{L:main computation} for more detailed internal references.

One technical innovation introduced along the way (see Appendix~\ref{sec:orbit lookup}) is a lightweight method for computing orbits of the action of a group $G$ on subsets of a set carrying a $G$-action; for instance, in the generic genus-7 case we compute orbits of 6-element sets of $\FF_2$-rational points on the 10-dimensional orthogonal Grassmannian. This construction may be of independent interest for other applications, including extending the tabulation of genus-$g$ curves over $\FF_2$ to a few larger values of $g$ for which the Brill--Noether stratification on moduli can again be made explicit (see Remark~\ref{rmk:Brill-Noether}), or finding supersingular genus-$g$ curves over $\FF_2$ for $g$ in a similar range. See \S\ref{sec:full census} for more discussion of the relevant issues.

As in \cite{part1} and \cite{part2}, the arguments depend on a number of computations in \SageMath{} \cite{sage} and \Magma{} \cite{sage};
the computations take about 8 hours on a single CPU (Intel i5-1135G7@2.40GHz) and can be reproduced using some
Jupyter notebooks found in the repository \cite{repo}. (Some functionality used in \SageMath{} is derived from \Gap{} \cite{gap} and \Singular{} \cite{singular}.)

\section{The structure of canonical curves}
\label{sec:canonical curves}

Let $C$ be a curve of genus $g$ over a \emph{finite} field $k$.
Let $\overline{k}$ be an algebraic closure of $C$.
We collect here a few facts about the geometry of $C$ that will be used frequently, and often without comment, in what follows. See \cite{saint-donat} for a characteristic-free treatment of much of this material (and \cite[\S 4.3]{griffiths-harris} for some additional details).

A \emph{$g^r_d$} on $C$ is a line bundle of degree $d$ whose space of global sections has dimension $r+1$; if such a bundle is basepoint-free, then it defines a degree-$d$ map $C \to \PP^r$. (If the bundle is not basepoint-free, then the global sections generate a basepoint-free subbundle of degree strictly less than $d$.)
Since $k$ is finite, every Galois-invariant divisor class on $C$ contains a $k$-rational divisor (see for example \cite[Remark~2.4]{castryck-tuitman}). Consequently, if $C_{\overline{k}}$ admits a \emph{unique} $g^r_d$ for some $r,d$, then so does $C$.

The Castelnuovo--Severi inequality (see for example \cite[Theorem~3.11.3]{stichtenoth})
asserts that if there exist curves $C_1, C_2$ of genera $g_1, g_2$ and morphisms $f_1\colon C \to C_1, f_2 \colon C \to C_2$ of degrees $d_1, d_2$ such that $k(C)$ is the compositum of $k(C_1)$ and $k(C_2)$ over $k$, then
\[
g \leq d_1 g_1 + d_2 g_2 + (d_1 - 1)(d_2 - 1).
\]
We will use this bound to ensure that certain low-degree maps out of $C$ occur in isolation.

We say that $C$ is \emph{hyperelliptic} if $C$ admits a $g^1_2$ (which is automatically basepoint-free if $g > 0$). By Castelnuovo--Severi, if $g > 1$ then $C_{\overline{k}}$ can admit only one $g^1_2$; consequently, $C$ is hyperelliptic if and only if $C_{\overline{k}}$ is hyperelliptic.
Let $\iota\colon C \to \PP^{g-1}_k$ be the \emph{canonical morphism}, defined by the canonical linear system; then $\iota$ is a degree-2 map onto a rational normal curve if $C$ is hyperelliptic and an embedding otherwise. By abuse of language, $\iota$ is commonly called the \emph{canonical embedding} even when $C$ is allowed to be  hyperelliptic.

For $g > 4$, we say that $C$ is \emph{trigonal} if it admits a $g^1_3$ but not a $g^1_2$ (so the former is necessarily basepoint-free). By Castelnuovo--Severi again,  $C_{\overline{k}}$ can admit only one $g^1_3$; consequently, $C$ is trigonal if and only if $C_{\overline{k}}$ is trigonal. By Petri's theorem (a/k/a the Max Noether--Enriques--Petri theorem), if $C$ is not trigonal or a smooth plane quintic (when $g=6$), then $\iota(C)$ is cut out by quadrics.

By contrast, for $C$ trigonal, the linear system of quadrics containing $\iota(C)$ cuts out a rational normal scroll; the latter is isomorphic to the Hirzebruch surface
\[
\mathbf{F}_n := \Proj_{\PP^1_k}(\calO_{\PP^1_k} \oplus \calO(n)_{\PP^1_k})
\]
for a certain integer $n \geq 0$ called the \emph{Maroni invariant}\footnote{Here we follow the terminology of \cite{saint-donat}. The original definition of Maroni \cite{maroni} follows a different numbering convention which is also commonly used.} of $C$.
The structure map $\mathbf{F}_n \to \PP^1_k$, whose fibers form a ruling of $\mathbf{F}_n$, restricts to the trigonal projection $\pi\colon C \to \PP^1_k$.
\begin{itemize}
\item
For $n>0$, $\mathbf{F}_n$ is isomorphic to an $(n,1)$-hypersurface in $\PP^1_k \times_k \PP^2_k$.
Let $b$ be the unique irreducible curve in $\mathbf{F}_n$ with negative self-intersection (the \emph{directrix})
and let $f$ be a fiber of the ruling; then 
\begin{equation} \label{eq:Fn intersections}
b^2 = -n, \qquad b\cdot f = 1, \qquad f^2 = 0,
\end{equation}
and blowing down $\mathbf{F}_n$ along $b$ yields the weighted projective space $\PP(1:1:n)_k$.
Of the linear systems
\begin{equation} \label{eq:Fn linear systems}
|3b + \tfrac{g + 3n+2}{2} f|, \qquad |b + \tfrac{g + n - 2}{2}f|, \qquad |{-2}b+(-n-2)f|,
\end{equation}
the first contains $C$, the second defines the embedding $\mathbf{F}_n \to \PP^{g-1}_k$, and the third is the canonical linear system.

\item
For $n=0$, we have $\mathbf{F}_{n,\overline{k}} \cong \PP^1_{\overline{k}} \times_{\overline{k}} \PP^1_{\overline{k}}$.
Let $b$ and $f$ be fibers of the two different rulings;
then \eqref{eq:Fn intersections} and the interpretation of \eqref{eq:Fn linear systems} remain valid.
Since $3 \neq \frac{g+2}{2}$, the symmetry of the two rulings is broken by $C$ and so everything descends from $\overline{k}$ to $k$.
\end{itemize}
Since $C$ and $b$ are effective, 
 $0 \leq b \cdot C = -3n +\frac{g+3n+2}{2}$, so
\[
0 \leq n  \leq \left\lfloor \tfrac{g+2}{3} \right\rfloor, \qquad n \equiv g \pmod{2}.
\]
In case $n = \tfrac{g+2}{3}$, we have $b \cdot C = 0$ and so $C$ also embeds into $\PP(1:1:n)$.

We say that $C$ is \emph{bielliptic} if it admits a degree-2 map to a genus-1 curve over $k$.
By Castelnuovo--Severi once more, if $g > 5$ then $C_{\overline{k}}$ can admit at most one such map;
consequently, $C$ is bielliptic if and only if $C_{\overline{k}}$ is bielliptic.

\section{Brill--Noether stratifications}

We now specialize the previous discussion to the genera of direct concern here, following Mukai. We use the conventions that the Grassmannian $\Gr(r, V)$ parametrizes subspaces of dimension $r$ of a specified vector space $V$ and that the Pl\"ucker embedding maps into $\PP(\wedge^r V^\vee)$.

\begin{theorem} \label{T:genus 6 strat}
Let $C$ be a curve of genus $6$ over a finite field $k$. Then one of the following holds.
\begin{enumerate}
\item
The curve $C$ is hyperelliptic.
\item
The curve $C$ is trigonal of Maroni invariant $2$. In this case, $C$ occurs as a complete intersection of type $(2,1) \cap (1,3)$ in $\PP^1_k \times_k \PP^2_k$, where the $(2,1)$-hypersurface is isomorphic to $\mathbf{F}_2$.
\item
The curve $C$ is trigonal of Maroni invariant $0$. In this case, $C$ occurs as a curve of bidegree $(3,4)$ in $\PP^1_k \times_k \PP^1_k$.
\item
The curve $C$ is bielliptic.
\item
The curve $C$ occurs as a smooth quintic curve in $\PP^2_k$.
\item
The curve $C$ occurs as a transverse intersection of four hyperplanes, a quadric hypersurface, 
and the $6$-dimensional Grassmannian $\Gr(2,5)$ in $\PP^9_k$.
\end{enumerate}
\end{theorem}
\begin{proof}
This again follows from Petri's theorem except for the last case, in which the description can be found in \cite[Theorem~5.2]{mukai-cg}.
We recall that argument both to fill in some details that are left to the reader in \cite{mukai-cg} (by comparison with a similar argument in genus 8),
and to see that it is characteristic-free and applies over a finite base field.

By the Brill--Noether theorem on the existence of special divisors (see \cite{kleiman-laksov} for a characteristic-free treatment), 
$C_{\overline{k}}$ admits a $g^1_4$, which we call $\xi$; let $\eta := \omega_C \xi^{-1}$ be its Serre adjoint. 
By Riemann--Roch we have 
\[
h^0(\xi) = 2, \qquad h^0(\eta) = h^0(\xi) + g-1-\deg(\xi) = 3;
\]
that is, $\eta$ is a $g^2_6$.
Since $C$ is not trigonal or a plane quintic, the linear system $|\eta|$ is basepoint-free;
we thus have a map $\Phi_{|\eta|}\colon C_{\overline{k}} \to \PP^2_{\overline{k}}$ induced by $\eta$.
The image of $\Phi_{|\eta|}$ cannot be a singular cubic or a smooth cubic because we are assuming $C$ is neither hyperelliptic nor bielliptic,
so it must be a sextic curve $\overline{C}$.
Other than $\xi$, every $g^1_4$ of $C_{\overline{k}}$ arises by projection from a double point of $\overline{C}$; it follows that the space
$W^1_4(C_{\overline{k}})$ parametrizing the $g^1_4$'s of $C_{\overline{k}}$ is finite (recovering \cite[Proposition~5.3]{mukai-cg}).

We now emulate \cite[Lemma~3.6]{mukai-cg}.
The extensions $0 \to \xi \to E \to \eta \to 0$ are parametrized by $\Ext(\eta, \xi) \cong H^1(\eta^{-1} \xi)$, which is Serre dual to
$H^0(\eta^2)$. For an extension $e$, let $\delta_e: H^0(\eta) \to H^1(\xi)$ be the corresponding connecting homomorphism; then the linear
map
\[
\Delta\colon \Ext(\eta, \xi) \to H_0(\eta)^\vee \otimes H^1(\xi), \qquad e \mapsto \delta_e
\]
is dual to the multiplication map
\[
\mu\colon H^0(\eta) \otimes H^0(\eta) \to H^0(\eta^2).
\]
By Riemann--Roch again, $h^0(\eta^2) = \deg(\eta^2) - g + 1 = 7$. 
Since the image of $\Phi_{|\eta|}$ cannot be contained in a conic,
the linear map $\Sym^2 H^0(\eta) \to H^0(\eta^2)$ is injective, so its cokernel has codimension 1.
We conclude that $\ker(\Delta)$ is one-dimensional, and so there is a unique nontrivial extension of $\eta$ by $\xi$ which is a stable bundle with five linearly independent global sections. Because $k$ is perfect, this uniqueness property ensures that the resulting vector bundle of rank 2 on $C_{\overline{k}}$ descends to a unique vector bundle $E$ on $C$. 
We have now recovered \cite[Theorem~5.1(1)]{mukai-cg}.
We deduce from this an analogue of \cite[Lemma~3.10]{mukai-cg}: if $\xi'$ is any $g^1_4$ on $C_{\overline{k}}$, then
$\dim \Hom(\xi', E) \leq 1$.

Since $\delta_e = 0$ and $\xi$ and $\eta$ are generated by global sections, so is $E$.
Hence for each point in $C$, the fiber of $E$ at this point is a 2-dimensional quotient of the 5-dimensional space $H^0(E)$; this defines a map $\Phi_{|E|}\colon C \to \Gr(2, H^0(E)^\vee)$. Let $\lambda\colon \wedge^2 H^0(E) \to H^0(\wedge^2 E) = H^0(\omega_C)$ be the natural map. We then have a commutative diagram
\[
\xymatrix{
C_{\overline{k}} \ar^{\Phi_{|E|}}[r] \ar[d] & \Gr(2, H^0(E)^\vee) \ar[d] \\
\PP(H^0(\omega_C)) \ar^{\PP(\lambda)}[r] & \PP(\wedge^2 H^0(E))
}
\]
where the left vertical arrow is the canonical embedding and the right vertical arrow is the Pl\"ucker embedding.
The hyperplanes of $\PP(\wedge^2 H^0(E))$ are parametrized by $\PP((\wedge^2 H^0(E))^\vee)$;
the hyperplanes among these which containing the image of $C$ are parametrized by $\PP((\ker \lambda)^\vee)$.

We now emulate \cite[Theorem~B]{mukai-cg}.
Suppose that $U \subset H^0(E)$ is a 2-dimensional subspace such that $\lambda(\wedge^2 U) = 0$.
Then the evaluation map $U \otimes \calO_C \to E$ is not generically surjective; its image is a line subbundle $L$ of $E$
satisfying $h^0(L) \geq 2$. The stability of $E$ forces $\deg(L) < 5$, and $h^0(L) = 2$ since $C$ has no $g^2_4$ by the adjunction formula.
Since $C$ is not hyperelliptic or trigonal, $L$ must be a $g^1_4$. That is, this construction defines a map from
$\PP((\ker \lambda)^\vee) \cap \Gr(2, H^0(E)^\vee)$ to $W^1_4(C_{\overline{k}})$; this map is injective by
our analogue of \cite[Lemma~3.10]{mukai-cg} and surjective by \cite[Proposition~3.1]{mukai-cg}.
We have now recovered \cite[Theorem~5.1(2)]{mukai-cg}.

We now follow the proof of \cite[Theorem~5.2]{mukai-cg} as written. To wit, since $\PP((\ker \lambda)^\vee) \cap \Gr(2, H^0(E)^\vee) \cong W^1_4(C_{\overline{k}})$ is finite
and $\Gr(2, H^0(E)^\vee)$ has codimension 3 in $\PP(\wedge^2 H^0(E))$,
$\dim(\ker \lambda) \leq 4$; hence $\lambda$ is surjective and so $\Phi_{|E|}$ is an embedding.
By Petri's theorem (\S\ref{sec:canonical curves}), 
the image of $\Phi_{|E|}$ is cut out by the hyperplanes in $\PP((\ker \lambda)^\vee)$ plus a single quadric.
\end{proof}

\begin{theorem} \label{T:Mukai class}
Let $C$ be a curve of genus $7$ over a finite field $k$. Then one of the following holds.
\begin{enumerate}
\item
The curve $C$ is hyperelliptic.
\item
The curve $C$ is trigonal of Maroni invariant $3$.
In this case, $C$ occurs as a hypersurface of degree $9$ in $\PP(1:1:3)_k$.
\item
The curve $C$ is trigonal of Maroni invariant $1$.
In this case, $C$ occurs as a complete intersection of type $(1,1) \cap (3,3)$ in $\PP^1_k \times_k \PP^2_k$.
\item
The curve $C$ is bielliptic.
\item
The curve $C$ is not bielliptic but admits a $g^2_6$ which is self-adjoint (squares to the canonical class).
In this case, $C$ is a complete intersection of type $(3) \cap (4)$ in $\PP(1 : 1:1:2)_k$,
where the degree $3$ hypersurface can be taken to be defined by $x_0x_3 + P_3(x_1, x_2) = 0$ for some separable cubic $P_3$.
\item
The curve $C$ admits a pair of distinct $g^2_6$'s. In this case,
$C$ occurs as a complete intersection of type $(1,1) \cap (1,1) \cap (2,2)$
in $\PP^2_k \times_k \PP^2_k$.
\item
The curve $C$ does not admit a $g^2_6$ but $C_{\overline{k}}$ does.
In this case, $C$ occurs as a complete intersection of type $(1,1) \cap (1,1) \cap (2,2)$
in the quadratic twist of $\PP^2_k \times_k \PP^2_k$.
\item
The curve $C$ admits a $g^1_4$ but $C_{\overline{k}}$ does not admit a $g^2_6$.
In this case, $C$ occurs as a complete intersection of type $(1,1) \cap (1,2) \cap (1,2)$ in $\PP^1_k \times_k \PP^3_k$
in which the $(1,1)$-hypersurface is a $\PP^2$-bundle over $\PP^1$. (It is also true that all of the $(1,2)$-hypersurfaces vanishing on $C$ are geometrically irreducible, but we won't use this here.)
\item
The curve $C$ does not admit a $g^1_4$.
In this case, $C$ occurs as a transverse intersection of $9$ hyperplanes and the orthogonal Grassmannian $\OG^+(5, 10)$ in $\PP^{15}_k$.
\end{enumerate}
\end{theorem}
\begin{proof}
Petri's theorem covers cases (1)--(3). We treat cases (4)--(8) as summarized in \cite[Table~1]{mukai},
postponing case (9) until \S\ref{sec:OG} where we introduce the relevant notation.

Suppose that $C_{\overline{k}}$ is not hyperelliptic or trigonal but admits a $g^1_4$; let $\xi$ be one such and let $\eta:= \omega_C \xi^{-1}$ be its Serre adjoint,
which by Riemann--Roch is a $g^2_8$. Since $C$ cannot admit a $g^2_5$, $|\eta|$ is basepoint-free. Let $\pi\colon C_{\overline{k}} \to \PP^1_{\overline{k}}$
and $\tau\colon C_{\overline{k}} \to \PP^3_{\overline{k}}$ be the maps defined by $|\xi|$ and $|\eta|$.

If $C_{\overline{k}}$ has no $g^2_6$, then $\tau$ is an embedding; its image cannot lie in a quadric by the adjunction formula, so
$\eta$ does not factor as a product of two $g^1_4$'s.
By \cite[Corollary~3.2]{mukai-cg}, any $g^1_4$ other than $\xi$ would have to occur as a subbundle of $\eta$,
so in fact $\xi$ is the unique $g^1_4$ on $C_{\overline{k}}$.
This means that both $\xi$ and $\eta$ descend from $C_{\overline{k}}$ to $C$.
We can now follow the proofs of \cite[Lemma~6.1, Proposition~6.3]{mukai} to the desired conclusion.

Suppose instead that $C_{\overline{k}}$ has a $g^2_6$; let $\alpha$ be one such and let $\beta$ be its Serre adjoint, which is also a $g^2_6$.
Since $C_{\overline{k}}$ is not hyperelliptic or trigonal, the map $f\colon C_{\overline{k}} \to \PP^2_{\overline{k}}$ defined by $|\alpha|$
is either birational onto a sextic or a double cover of a smooth cubic. In the latter case $C_{\overline{k}}$ is evidently bielliptic, as then is $C$.
In the former case, from the proof of \cite[Proposition~3.1]{mukai-cg} we see that there are no $g^2_6$'s on $C_{\overline{k}}$ other than $\alpha$ and $\beta$. Namely, if $\zeta$ is a third $g^2_6$, then for $E = \xi \oplus \eta$ we have $h^0(\zeta^{-1}E) = 0$ and so 
\[
h^0(\zeta \xi) + h^0(\zeta \eta) = h^0(\zeta E) = h^0(\omega_C \zeta E^{\vee}) = h^0(\zeta^{-1}E) + 2 \deg(\zeta) = 12;
\]
this is only possible if one of $\zeta \xi$ or $\zeta \eta$ is special, which is impossible because they are both of degree 12 and not canonical.

If $\alpha$ and $\beta$ are isomorphic, then they both descend to $C$;
otherwise, they descend either to $C$ or to its quadratic base extension.
We can now follow the proof of \cite[Proposition~6.5]{mukai} to conclude.
\end{proof}

\begin{remark}
In \cite[Proposition~6.4]{mukai}, it is also shown that  bielliptic curves occur as complete intersections of type $(3) \cap (4)$ in $\PP(1:1:1:2)_k$. We will not use this in our computations.
\end{remark}

\begin{remark} \label{rmk:Brill-Noether}
While we will not need to do so here, it is possible to push this treatment through to a few higher genera.  For example, Mukai showed that (over an algebraically closed field) a genus-$8$ curve with no $g^2_7$ is a linear section of the $8$-dimensional Grassmannian $\Gr(2,6) \subset \PP^{14}$ \cite{mukai}; building on this, the complete Brill--Noether stratification in genus 8 has been described by Ide--Mukai \cite{mukai-ide}.
Similarly, a genus-$9$ curve with no $g^1_5$ is a linear section of the 6-dimensional symplectic Grassmannian $\SpG(3, 6) \subset \PP^{13}$ \cite{mukai2}, and an analogous assertion holds in genus 10 \cite{mukai10}. Pushing this even further would amount to establishing unirationality of the moduli space of genus-$g$ curves, which is known to hold for $g \leq 14$ \cite{verra} and to fail for $g \geq 22$
\cite{eisenbud-harris}, \cite{farkas}.
\end{remark}

\section{Overview of the proof}

We now give an overview of the proof of Theorem~\ref{T:main}.

\begin{lemma} \label{L:main computation}
For the various strata in moduli described above,
the number of isomorphism classes of curves $C$ over $\FF_2$ in each stratum
admitting \'etale double coverings $C' \to C$ such that $\#J(C')(\FF_2) = \#J(C)(\FF_2)$
 is given in Table~\ref{table:stratification}. In particular, Theorem~\ref{T:main} holds.
\end{lemma}
\begin{table}[ht]
\Small
\caption{Outline of the use of the Brill--Noether stratification in the proof of Lemma~\ref{L:main computation}. Of the columns, ``Dim'' records the dimension of the stratum in moduli, ``See'' locates the description of this case in the text, ``$\#C$'' counts curves whose point counts appear in Table~\ref{table:genus 567 counts}, and ``$\#C'$'' counts double covers with relative class number 1.}
\begin{tabular}{c||c|c|c|c|c||c|c|c||c|c}
& \multicolumn{5}{c||}{$g=6$} & \multicolumn{5}{c}{$g=7$} \\
Type of $C$ & Dim & See & $\#C$ & $\#C'$ & Time & Dim & See & $\#C$ & $\#C'$& Time\\
\hline
hyperelliptic & 11 & \S\ref{sec:point counts} & 0 &0 & --- &13 & \S\ref{sec:point counts} & 0 & 0 & ---\\
trigonal, Maroni $\geq 2$ & 12 & \S\ref{sec:orbit calc} & 4& 0& 10m & 13 & \S\ref{sec:point counts} & 0 & 0 & --- \\
trigonal, Maroni $\leq 1$ & 13 & \S\ref{sec:orbit calc} & 9& 0& 2m & 15 & \S\ref{sec:orbit calc} & 0 & 0 & 5m \\
bielliptic & 10 & \S\ref{sec:point counts} & 0 & 0 &---& 12 & \S\ref{sec:point counts} & 2 & 1 & 5m\\
plane quintic &12 & \S\ref{sec:orbit calc} &1&0& 1m & --- & --- & --- & --- & --- \\
self-adjoint $g^2_6$ &--- &---&---& ---& ---&15 & \S\ref{sec:orbit calc} & 0 & 0 & 5m\\
rational $g^2_6$ pair &--- &---&--- &---&---&16&\S\ref{sec:orbit calc} & 0 & 0 & 30m\\
irrational $g^2_6$ pair &---& --- & --- &---&---&16& \S\ref{sec:orbit calc} & 0 & 0 & 45m\\
tetragonal, no $g^2_6$ &---&---&--- &---& ---&17& \S\ref{sec:orbit calc} & 1 & 0 & 2h\\
generic &15& \S\ref{sec:orbit calc} & 38 & 2 & 2.5h&18& \S\ref{sec:OG} & 1 & 0 & 1h
\end{tabular}
\label{table:stratification}
\end{table}
\begin{proof}
To begin with, we recall from \cite[Theorem~1.3(b)]{part1} that the Weil polynomials of $C$ and $C'$ are restricted to an explicit finite list. In Table~\ref{table:genus 567 counts}, we list the possible values of the tuple $(\#C(\FF_{2^i}))_{i=1}^g$.

\begin{table}[ht]
\Small
\begin{tabular}{ccc}
4, 14, 16, 18, 14, 92 & 5, 11, 11, 31, 40, 53 & 6, 10, 9, 38, 11, 79\\
4, 14, 16, 18, 24, 68 & 5, 11, 11, 31, 40, 65 & 6, 10, 9, 38, 21, 67\\
4, 14, 16, 26, 14, 68 & 5, 11, 11, 39, 20, 53 & 6, 10, 9, 38, 31, 55\\
4, 16, 16, 20, 9, 64 & 5, 11, 11, 39, 20, 65 & 6, 14, 6, 26, 26, 68\\
5, 11, 11, 31, 20, 65 & 5, 13, 14, 25, 15, 70 & 6, 14, 6, 26, 26, 80\\
5, 11, 11, 31, 20, 77 & 5, 13, 14, 25, 15, 82 & 6, 14, 6, 26, 36, 56\\
5, 11, 11, 31, 20, 89 & 5, 13, 14, 25, 15, 94 & 6, 14, 6, 34, 16, 56\\
5, 11, 11, 31, 30, 53 & 5, 13, 14, 25, 25, 46 & 6, 14, 6, 34, 26, 44\\
5, 11, 11, 31, 30, 65 & 5, 13, 14, 25, 25, 58 & 6, 14, 12, 26, 6, 44\\
5, 11, 11, 31, 30, 77 & 5, 13, 14, 25, 25, 70 & 6, 14, 12, 26, 6, 56\\
5, 11, 11, 31, 30, 89 & 5, 15, 5, 35, 20, 45 & 6, 14, 12, 26, 6, 66
\end{tabular}
\qquad
\begin{tabular}{c}
6, 18, 12, 18, 6, 60, 174 \\
6, 18, 12, 18, 6, 72, 132 \\
6, 18, 12, 18, 6, 84, 90 \\
7, 15, 7, 31, 12, 69, 126 \\
7, 15, 7, 31, 22, 45, 112 \\
7, 15, 7, 31, 22, 57, 70 \\
7, 15, 7, 31, 22, 57, 84 
\end{tabular}

\medskip
\caption{Tuples $(\#C(\FF_{2^i}))_{i=1}^{g}$ allowed by \cite[Theorem~1.3(b)]{part1} for $g = 6,7$.}
\label{table:genus 567 counts}
\end{table}

For each stratum, we exhibit a set $T$ of schemes of finite type over $\FF_2$ of size at most $10^6$, such that every curve $C$ over $\FF_2$ belonging to the specified stratum whose point counts are consistent with Table~\ref{table:genus 567 counts} is isomorphic to some scheme in $T$.
In most cases, all of the schemes in $T$ will be presented as subschemes of a single ambient scheme $X$;
see Table~\ref{table:stratification} for internal cross-references.

Given a set $T$ as indicated, we conclude as follows (iterating over all $C \in T$).
 All computations are done in \SageMath{} except as indicated.
 \begin{itemize}
\item
Optionally, for one or more $i \geq 1$, compute $\#C(\FF_{2^i})$ using a lookup table of $X(\FF_{2^i})$, retaining cases consistent with Table~\ref{table:genus 567 counts}. We typically do this when we have at least $10^5$ cases to deal with.
\item
Optionally, for one or more $i \geq 1$, compute $\#C(\FF_{2^i})$ by computing the length of the intersection in $C \times_{\FF_2} C$ of the diagonal with the graph of the $i$-th power of the Frobenius morphism, retaining cases consistent with Table~\ref{table:genus 567 counts}.
We typically do this when we have between $10^4$ and $10^5$ cases to deal with.
\item
Use \Magma{} to check whether $C$ is one-dimensional and integral, and if so whether its normalization has genus $g$. If so, compute $\#C(\FF_{2^i})$ for $i=1,\dots,g$ by enumerating places of the function field of $C$, retaining cases consistent with Table~\ref{table:genus 567 counts}.
\item
Use \Magma{} to compute isomorphism class representatives among the remaining curves. The count of these is reported in Table~\ref{table:stratification}.
\item
Use \Magma{} to identify quadratic extensions of the remaining function fields with relative class number 1. The count of these is reported in Table~\ref{table:stratification};
this yields the claimed results.
\qedhere
\end{itemize}
\end{proof}

Table~\ref{table:stratification} also includes in each case a rough timing of the computation.
The timings should not be taken too seriously; they reflect some combination of the dimensions of the strata in moduli (included in Table~\ref{table:stratification}), the special nature of the Weil polynomials in question (which we exploit especially heavily for generic curves of genus 7), the highly nonuniform extent to which we attempted to optimize the calculation in the various cases, the imbalance between genus 6 and 7 in Table~\ref{table:genus 567 counts},
and variable load on the machine in question.

\section{Point counts}
\label{sec:point counts}

In a few cases of Lemma~\ref{L:main computation}, we can confirm that the options listed in Table~\ref{table:genus 567 counts} imply a nontrivial lower bound on the gonality of $C$.
This amounts to settling some cases of Lemma~\ref{L:main computation} with $T = \emptyset$.

\begin{itemize}
\item
If $g=6$, then $C$ cannot be hyperelliptic: we have $\#C(\FF_4) > 10 = 2\#\PP^1(\FF_4)$ except in three cases where
$\#C(\FF_{16}) = 38 > 2\#\PP^1(\FF_{16})$.
\item
If $g=7$, then $C$ cannot be hyperelliptic: we have $\#C(\FF_4) \geq 15 > 2\#\PP^1(\FF_4)$.
\item
If $g=7$ and $\#C(\FF_2) = 6$, then $C$ cannot be trigonal: we have $\#C(\FF_4) = 18 > 15 = 3\#\PP^1(\FF_4)$.
\item
If $g=7$ and $\#C(\FF_2) = 7$, then $C$ cannot be trigonal of Maroni invariant 3: we have $\#C(\FF_2) = 7$
which exceeds the number of \emph{smooth} points of $\PP(1:1:3)(\FF_2)$. 
\end{itemize}

We can use similar logic in the case where $C$ is bielliptic.
Suppose that $C \to E$ is a double covering of an elliptic curve. 
Then the Weil polynomial of $E$ must divide that of $C$, and moreover must satisfy the resultant criterion \cite[Corollary~9.4]{part1}. For $g=6$, the possibilities are listed in 
Table~\ref{table:genus 6 bielliptic}; in most cases, we find that $\#C(\FF_{2^i}) > 2 \#E(\FF_{2^i})$ for some $i$, an impossibility. In one case, $\#C(\FF_4) = 2\#E(\FF_4)$, which ensures that $C \to E$ does not ramify over any degree-1 places, but this is inconsistent with the fact that $\#C(\FF_2) \not\equiv 0 \pmod{2}$. (Alternatively, the unique degree-3 place of $E$ must map to a degree-1 place of $C$, which again contradicts  $\#C(\FF_4) = 2\#E(\FF_4)$.) We thus again settle this case of Lemma~\ref{L:main computation} with $T = \emptyset$.

\begin{table}
\begin{tabular}{c|c|c}
$\#E(\FF_{2^i})_{i=1}^4$ & $\#C(\FF_{2^i})_{i=1}^4$ & Disposition \\
\hline
$(1, 5,13,25)$ & $(6,10,9,38)$ & $\#C(\FF_2) >2\#E(\FF_2)$ \\
$(3,9,9,9)$ & $(5,13,41,25)$ & $\#C(\FF_{16}) > 2\#E(\FF_{16})$ \\
$(3,9,9,9)$ & $(6,10,9,38)$ & $\#C(\FF_{16}) > 2\#E(\FF_{16})$ \\
$(5,5,5,25)$ & $(5,13,14,25)$ & $\#C(\FF_{4}) > 2\#E(\FF_{4})$ \\
$(5,5,5,25)$ & $(6,10,9,38)$ & $\#C(\FF_{4}) = 2\#E(\FF_{4})$, $\#C(\FF_2) \not\equiv 0 \pmod{2}$ \\
\end{tabular}
\medskip
\caption{Possible point counts for $C$ bielliptic of genus 6 covering the genus-1 curve $E$.}
\label{table:genus 6 bielliptic}
\end{table}

For $g=7$, we may make a similar application of the resultant criterion to see that $\#C(\FF_2) = 6$
and $\#E(\FF_2) \in \{3,5\}$. We can rule out $\#E(\FF_2) = 5$ by noting that $\#C(\FF_4) = 18 > 10 = 2\#E(\FF_4)$;
we must thus have $\#E(\FF_2) = 3$. Now note that $E$ has $p$-rank 0 and $C$ has $p$-rank 5, so 
by the Deuring--Shafarevich formula \cite[(7.2)]{part1} the map $E \to C$ must ramify over six distinct geometric points.
Since $\#C(\FF_4) = 18 = 2\#E(\FF_4)$, the map $C \to E$ cannot ramify over any degree-1 or degree-2 points of $C$;
the ramification is thus either over a single degree-6 place or over the two distinct degree-3 places of $C$. We may thus settle this case of Lemma~\ref{L:main computation} by computing the set $T$ of double covers of $E$ with the indicated ramification divisors using \Magma.

\begin{remark}
Although we did not exploit this systematically in our calculations,
we point out that for every entry of Table~\ref{table:genus 567 counts} with $g=7$,
\cite[Theorem~4.15]{howe2021} implies the existence of a map from $C$ to a particular elliptic curve of degree at most 5. For example, when $\#C(\FF_{2^i})_{i=1}^7 = (6,18,12,18,6,72,132)$, $C$ must admit a degree-2 map to the elliptic curve $E$ with $\#E(\FF_2) = 3$; consequently, this option can be ignored in all but the bielliptic case.
\end{remark}

\section{The use of orbit lookup trees}
\label{sec:orbit calc}

In most of the remaining cases, we use a uniform paradigm to make an exhaustive calculation over the relevant term of the Brill--Noether stratification. 
Again, all computations are done in \SageMath{} except as indicated.

\begin{itemize}
\item
Let $X$ be the ambient variety indicated in Table~\ref{table:orbit setup}.
Compute the set $S := X(\FF_2)$ and the group $G := \Aut(X)(\FF_2)$.
\item
Use the method of orbit lookup trees (Appendix~\ref{sec:orbit lookup}) to compute orbit representatives for the action of $G$ on subsets of $S$ of size up to $g$. In some cases, we can impose extra conditions on the set $S$.
\begin{itemize}
\item
For $g=7$ with a rational $g^2_6$, no three points of $S$ have the same projection onto either $\PP^2_k$.
\item
For $g=7$ tetragonal, no five points of $S$ have the same projection in $\PP^1_k$.
\end{itemize}
\item
For each orbit representative for subsets of size in $\{4, 5, 6\}$ (if $g=6$) or $\{6,7\}$ (if $g=7$), use linear algebra to find all tuples of hypersurfaces $X_1,\dots,X_{m-1}$ of the indicated degrees passing through these $\FF_2$-points.
In the case of $g=7$ trigonal of Maroni invariant 1, we require $X_1$ to be smooth.
\item
For each choice, impose linear conditions on $X_m$ to ensure that $X_1 \cap \cdots \cap X_m$ has \emph{exactly} the specified set of $\FF_2$-rational points. (This crucially exploits the fact that the base field is $\FF_2$; a similar strategy is used in \cite[\S 6]{faber-grantham}.)
Take $T$ to be the resulting set of schemes $X_1 \cap \cdots \cap X_m$.
\end{itemize}

See Table~\ref{table:orbit setup} for how the notation maps to the various cases. Some additional clarifications:
\begin{itemize}
\item
In the case of $g=6$ trigonal of Maroni invariant 2, we take $X = X_{2,1}$ to be defined by $(x_0^2 +x_1^2)y_1 + x_0 x_1 y_2$.
\item
In the case of $g=6$ generic, we find candidates for the intersection of type $(1)^4$ by computing orbits for the action on sets of 4 $k$-points of the dual of $\PP^9_k$.
We then apply generators of $\GL(4, \FF_2)$ to these subsets to identify cases where the
linear spans are $G$-equivalent (compare Remark~\ref{R:linear case}); 
this yields 20 candidates for $X_1 \cap \cdots \cap X_{m-1}$.
We finally enumerate subsets of $X \cap X_1\cap \cdots \cap X_{m-1}$ of size in $\{4,5,6\}$, without further use of the group action.
\item
In the case of $g=7$ with a self-adjoint $g^2_6$, we take $X = X_{3}$ to be defined by a polynomial of the form $x_0x_3 + P(x_1, x_2)$ with
\[
P \in \{x_1x_2(x_1+x_2), x_1(x_1^2+x_1x_2+x_2^2), x_1^3+x_1x_2^2 + x_2^3\}
\]
and ignore the group action.
\item
In the case of $g=7$ tetragonal, we take $X = X_{1,1}$ to be defined by $x_0 y_0 + x_1 y_1$. We then break symmetry when choosing the defining polynomials $P_1, P_2$ of $X_1, X_2$ by fixing a total ordering on the quotient of the space of $(1,2)$-polynomials by the multiples of $x_0 y_0 + x_1 y_1$ and then forcing an ordering on the classes of $P_1, P_2$.
\end{itemize}

The generic case in genus 7 is handled slightly differently to avoid the computational bottleneck of enumerating orbits of 7-element subsets of $X$; see \S\ref{sec:OG}.

\begin{table}[ht]
\begin{tabular}{c|c|c|c|c}
$g$ & Case & $X$ & $X_1,\dots,X_{m-1}$ & $X_m$ \\
\hline
6 & trigonal, Maroni 2 & $X_{2,1} \subset \PP^1 \times \PP^2$ & $\emptyset$ & $(1,3)$ \\
6 & trigonal, Maroni 0 & $\PP^1 \times \PP^1$ & $\emptyset$ & $(3,4)$ \\
6 & plane quintic & $\PP^2$& $\emptyset$ & $(5)$\\
6 & generic & $\Gr(2, 5) \subset \PP^9$ & $(1)^4$ & $(2)$ \\
\hline
7 & trigonal, Maroni 1 & $\PP^1 \times \PP^2$ & $(1,1)$ & $(3,3)$ \\
7 & self-adjoint $g^2_6$ & $X_3 \subset \PP(1:1:1:2)$ & $\emptyset$ & $(4)$ \\
7 & rational $g^2_6$ & $\PP^2 \times \PP^2$ & $(1,1)^2$ & $(2,2)$ \\
7 & irrational $g^2_6$ & twist of $\PP^2 \times \PP^2$ & $(1,1)^2$ & $(2,2)$ \\
7 & tetragonal & $X_{1,1} \subset \PP^1 \times \PP^3$ & $(1,2)$ & $(1,2)$ \\
\hline
7 & generic & $\OG^+ \subset \PP^{15}$ & $(1)^8$ & $(1)$
\end{tabular}
\medskip
\caption{Group actions associated to Brill--Noether strata.}
\label{table:orbit setup}
\end{table}

\section{The generic case in genus 7}
\label{sec:OG}

We now describe a variant of the paradigm from \S\ref{sec:orbit calc} to handle generic (non-tetragonal) curves of genus 7.
In the process, we summarize the proof of \cite[Theorem~0.4]{mukai} and so confirm case (9) of Theorem~\ref{T:Mukai class}.

Let $k$ be a finite field (of any characteristic).
Let $V$ be the vector space $k^{10}$ equipped with the quadratic form $\sum_{i=1}^5 x_i x_{5+i}$.
We write $\SO(V)$ for the unique index-2 subgroup of the orthogonal group of $V$;
it admits a characteristic-free characterization as 
the kernel of the \emph{Dickson invariant}.

The \emph{orthogonal Grassmannian} of $V$, denoted $\OG$, parametrizes Lagrangian (isotropic 5-dimensional) subspaces of $V$.
Let $L_0$ be the subspace spanned by the first 5 coordinate vectors $e_1,\dots,e_5$, which by construction is isotropic.
Then $\OG$ splits into two connected components, each of which parametrizes Lagrangian subspaces of $V$ whose
intersection with $L_0$ has a specified parity. Let $\OG^+$ be the component containing $L_0$; it carries an action of $\SO(V)$.

The space $\OG^+$ admits an analogue of the Pl\"ucker embedding called the \emph{spinor embedding}.
The target of the spinor embedding can be described as the projectivization of the even orthogonal algebra $\wedge^{\mathrm{ev}} L_0$. The spinor embedding can be computed easily using the following (characteristic-free) recipe described in \cite[\S 1]{mukai}.
Let $L_\infty$ be the subspace spanned by $e_6,\dots,e_{10}$; this is an isotropic subspace lying on the other component of $\OG$. Split the orthogonal algebra $S = \wedge^\bullet L_\infty$ into even and odd components $S^{\mathrm{ev}}, S^{\mathrm{odd}}$.
The \emph{Clifford map} $V \to \End S$ then carries each element of $L_\infty$ to a creation operator (taking the wedge product with that element) and each element of $L_0 \cong L_\infty^\vee$ to an annihilation operator (contraction).
For each Lagrangian subspace $L$, the Clifford map carries the elements of $L$ to endomorphisms of $S$ whose joint kernel is one-dimensional, and is contained in either $S^{\mathrm{ev}}$ or $S^{\mathrm{odd}}$ according to whether or not $L \in \OG^+$; this yields the spinor embedding.

The following combines \cite[Proposition~1.16, Proposition~2.2]{mukai}.
\begin{lemma}[Mukai] \label{L:section to canonical genus 7}
Let $P \subset \PP^{15}_k$ be a $6$-dimensional linear subspace which passes through $L_0$
and meets $\OG^+$ transversely.
\begin{enumerate}
\item[(a)] 
No $4$ points of $C = \OG^+ \cap P$ lie in a $2$-plane.
\item[(b)]
The scheme $C$ is a canonical curve of genus $7$ with one marked point with no $g^1_4$.
\end{enumerate}
\end{lemma}

Although for our purposes we only need the opposite implication to Lemma~\ref{L:section to canonical genus 7}, this direction is actually crucial to the argument.

Now let $C$ be a curve of genus 7 with one marked point admitting no $g^1_4$.
We set notation as in \cite[\S 3]{mukai}.
Let $C \subset \PP^6_k$ be the canonical embedding and set $W := H^0(\PP^6_k, I_C(2))$; by Petri's theorem (\S\ref{sec:canonical curves}), $\dim W  =(g-2)(g-3)/2 = 10$.
Set $E := N^\vee_{C/\PP^6_k} \otimes \omega_C^2$; this is a bundle of rank $g-2=5$. From the exact sequence
\[
0 \to N^\vee_{C/\PP^6_k} \to \left. \Omega_{\PP^6_k} \right|_C \to \omega_C \to 0
\]
we see that $\det E \cong \omega_C^2$. Since $N^\vee_{C/\PP^6_k} \cong I_C/I_C^2$ and $\omega_C \cong \calO_C(1)$, we obtain a linear map $W \to H^0(C, E)$.
Since $C$ is not trigonal, by Petri's theorem again, for every closed point $p \in C$ the kernel $W_p$ of the induced map from $W$ to the fiber $E_p$
is 5-dimensional; this defines a map $C \to \Gr(5, W)$. 

The following is \cite[Proposition~3.3]{mukai}.
\begin{lemma} \label{L:no large intersection}
With notation as above (i.e., $C$ is a curve of genus $7$ with no $g^1_4$),
for any two distinct closed points $p,q \in C$,
the intersection $W_p \cap W_q$ is $1$-dimensional (not just odd-dimensional).
\end{lemma}

The following is \cite[Theorem~4.2]{mukai}.
\begin{lemma}[Mukai]
With notation as above (i.e., $C$ is a curve of genus $7$ with no $g^1_4$),
let $f\colon \Sym^2 W \to H^0(C, \Sym^2 E)$ be the natural map.
Then every nonzero element of $\ker f$ is nondegenerate.
Consequently, $\dim \ker f \leq 1$.
\end{lemma}

When $C$ arises as in Lemma~\ref{L:section to canonical genus 7},
then we have a natural identification $V \cong W$ 
(see \cite[Corollary~2.5]{mukai}) and so the quadratic form on $V$
defines a nonzero element of $\ker f$. The upshot of this is that the embedding
$C \to \Gr(5,W)$ factors through $\OG^+$. As in \cite[(5.1)]{mukai},
this defines an injective map from the space of linear sections of $\OG^+$ passing through $L_0$ to the moduli space of genus 7 curves with one marked point. As these spaces are both 11-dimensional and the latter is irreducible \cite{deligne-mumford}, the map is dominant;
that is, the \emph{generic} curve of genus 7 with one marked point occurs as a linear section of $\OG^+$ passing through $L_0$.
As in \cite[Corollary~5.3]{mukai}, it then follows that $\dim \ker f = 1$ in all cases
(where we still assume that $C$ has no $g^1_4$). Thus we end up with a map
$C \to \Gr(5, W)$, and it is straightforward to check that it is an embedding
\cite[Theorem~0.4]{mukai}.
This establishes case (9) of Theorem~\ref{T:Mukai class} modulo the following remark.

\begin{remark}
We comment briefly on what happens if we consider a curve of genus 7 \emph{without} a marked point. In this case, $H^0(E)$ still carries a distinguished nondegenerate quadratic form, but it is not guaranteed to be isomorphic \emph{over $k$} to the form we are using.

When $k$ is finite, however, this issue does not arise for the following reason. There are 
only two isomorphism classes of such forms, distinguished by the discriminant in odd characteristic and the Arf invariant in even characteristic; moreover, these become isomorphic over the quadratic extension of $k$. By passing to a finite extension of $k$ of suitably large \emph{odd} degree over which $C$ acquires a rational point, we can see that $H^0(E)$ must carry the quadratic form which admits Lagrangian subspaces over $k$.
\end{remark}

We now specialize the previous discussion to the case $k = \FF_2$ and describe the computation that proves Lemma~\ref{L:main computation} for curves of genus 7 with no $g^1_4$.
We first build an orbit lookup tree (Appendix~\ref{sec:orbit lookup}) to depth 6 for the action of $G := \SO(V)$ on $S := \OG^+(\FF_2)$ with the following tuples forbidden.
\begin{itemize}
\item
Any pair of points corresponding to Lagrangian subspaces whose intersection has dimension greater than 1 (ruled out by Lemma~\ref{L:no large intersection}).
\item
Any triple of collinear points or 4-tuple of coplanar points (ruled out by Lemma~\ref{L:section to canonical genus 7}).
\item
Any tuple whose linear span has positive-dimensional intersection with $\OG^+$
(ruled out by the fact that the canonical embedding does not factor through a hyperplane).
\item
Any tuple whose linear span meets $\OG^+$ in 8 or more $\FF_2$-rational points
(ruled out because Table~\ref{table:genus 567 counts} requires $\#C(\FF_2) \leq 7$).
\end{itemize}
This yields 494 orbit representatives.
Inspecting the results, we find that each orbit representative spans either a 4-plane or a 5-plane in $\PP^{15}_k$.

Let $V$ be an orbit representative.
We now separate into cases depending on whether $\#C(\FF_2) = 6$ or $\#C(\FF_2) = 7$, whether or not $V$ spans a 4-plane or a 5-plane, and whether or not this span contains any more points of $\OG^+$.
\begin{itemize}
\item
If $\#C(\FF_2) = 7$ and $V$ spans a 4-plane, we verify that the span of $V$ does not contain a seventh $\FF_2$-point of $\OG^+$. This means that without loss of generality, we may ignore this case.
\item
If $\#C(\FF_2) = 6$, $V$ spans a 5-plane, and this 5-plane does not contain a 7th point of $\OG^+$, we hash the remaining $\FF_2$-points of $\OG^+$ according to their joint linear span with $V$, retaining 6-planes that do not appear at all.
\item
If $\#C(\FF_2) = 7$, $V$ spans a 5-plane, and this 5-plane does not contain a 7th point of $\OG^+$, we build a similar hash table, retaining 6-planes that occur exactly once.
\item
If $\#C(\FF_2) = 6$ and $V$ spans a 4-plane, we build a similar hash table, retaining 5-planes that do not appear at all.
We then hash pairs of 5-planes in this list according to their span, retaining 6-planes that appear 3 times.
\item
If $\#C(\FF_2) = 7$, $V$ spans a 5-plane, and this 5-plane contains a 7th point of $\OG^+$, we may assume without loss of generality that the maximum intersection multiplicity of the linear span with $\OG^+$ among the $\FF_2$-rational intersection points occurs for the 7th point.
We then build a similar hash table, retaining 6-planes that do not appear at all.
\end{itemize}

We then take $T$ to be the set of intersections of the resulting 6-planes with $\OG^+$.

\section{Towards a full census in genera 6 and 7}
\label{sec:full census}

It would be extremely desirable to refine the methods used here to complete a full census of curves of genus 6 and 7 over $\FF_2$, both to provide a consistency check on our own work and to make the tables available for other purposes. One important aspect of such a census is the process of making it simultaneously reliable and rigorous. In other words, given a putative list of isomorphism classes of curves of genus $g$ over $\FF_2$, how can one verify that this list is accurate?

In one direction, it is easy from the data first to compute individual properties of the curves in question, such as their automorphism groups, and to test pairs of curves to confirm that they are not isomorphic; the latter can be accelerated by first hashing curves according to their zeta function (and any other data that may have been computed, such as the order and structure of the automorphism group). 

In the other direction, one may check completeness of the list by computing the count of $\FF_q$-points on the moduli space $\calM_g$ using the Lefschetz trace formula
for the moduli space $\overline{\calM}_g$ of stable curves of genus $g$ and the combinatorics of the boundary strata.
(Note that the point count is ``stacky'' in that each equivalence class of $\FF_q$-points is weighted inversely by the order of its automorphism group.)
For $g=6$, it is known that $\#\calM_g(\FF_q)$ equals a fixed polynomial in $q$ \cite[Corollary~1.6]{canning-larson},
which in principle can be computed using the \SageMath{} package described in \cite{admcycles}.
For $g=7$, while $\#\overline{\calM}_g(\FF_q)$ is known to equal a fixed polynomial in $q$ \cite[Corollary~1.5]{canning-larson}, 
the same does not immediately follow for $\#\calM_g(\FF_q)$
due to a possible contribution from the level-1 modular form $\Delta$ in the boundary of $\overline{\calM}_7$.

As noted in the introduction, this discussion in principle also applies in some higher genera, 
using the parametrizations indicated in Remark~\ref{rmk:Brill-Noether}. However, it is not clear to us to what extent the resulting computations are feasible.

\appendix

\section{Orbit lookup trees}
\label{sec:orbit lookup}

Throughout this appendix, fix a finite group $G$ and a finite set $S$ equipped with a left $G$-action.
We exhibit a combinatorial structure that allows us to efficiently compute orbit representatives for the action of $G$ on $k$-element subsets of $S$ for various small values of $k$. Such computations are already implemented in software (notably in \Magma); however, the approach we take here seems to be well-adapted to our present work, as it avoids instantiating in memory the entire set of $k$-element subsets of $S$.

\begin{defn}
Let $\Gamma$ be a directed graph with loops with vertex set $S \sqcup \{\bullet\}$,
in which each edge is either of the form $\bullet \to v$ for some $v \in S$,
or of the form $v_1 \stackrel{g}{\to} v_2$ for some $v_1, v_2 \in S$ with a label $g \in G$
satisfying $g(v_1) = v_2$.
Defining connected components of $\Gamma$ in terms of the underlying undirected graph,
we say that a component is \emph{eligible} if it does not contain $\bullet$ and a vertex is \emph{eligible} if it lies in an eligible component.

A \emph{group retract} of $\Gamma$ consists of a subset $V$ of $\Gamma$ consisting of one vertex in each eligible connected component, together with a function $h$ from the union of the eligible components to $G$ with the property that for each $v' \in \Gamma$
in the connected component of $v \in V$, we have $h(v')(v) = v'$.

One way to compute a group retract, given a choice of $V$, is to fix a spanning tree in
each eligible component, then compute the unique function satisfying:
\begin{itemize}
\item
for all $v \in V$, $h(v) = 1_G$;
\item
for each chosen spanning tree, for every edge $v_1 \stackrel{g}{\to} v_2$ in the tree, $h(v_2) = g h(v_1)$.
\end{itemize}
Note that this function can be computed in linear time in the input length.
\end{defn}

\begin{defn} \label{D:orbit lookup tree}
Let $F$ be a subset of the power set of $S$ (the \emph{forbidden subsets}).
We say that a subset of $S$ is \emph{eligible} if it contains no forbidden subset.

For any positive integer $n$, an \emph{orbit lookup tree} of depth $n$ 
(for $G,S,F$)
is a rooted tree $T_n$ of depth $n$ with the following properties (and additional data as indicated).
\begin{itemize}
\item
Each node at depth $k$ is labeled by a $k$-element subset $U$ of $S$,
colored either red or green.
In what follows, we freely conflate nodes with their labels.
\item
The parent of every node $U$ is a green node which is a subset of $U$.
In particular, there is a unique ordering of the elements $x_1,\dots,x_k$ of $U$ such that 
each initial
segment of this sequence is also a node; we write $U = [x_1,\dots,x_k]$ instead of $U = \{x_1,\dots,x_k\}$ when we need to indicate this choice of ordering.
\item
For $k=0,\dots,n$,
the green nodes at depth $k$ form a set of $G$-orbit representatives for the eligible $k$-element subsets of $S$.
\item
For each eligible node $U$, we further record an element $g_U \in G$ such that $g_U^{-1}(U)$ is a green node.
\item
For each green node $U$, we further record the stabilizer $G_U$.
\item
Every green node $U$ at depth $k < n$ has children which form
a set of $G_U$-orbit representatives of the $(k+1)$-element subsets of $S$ containing $U$.
We further record a function $h_U\colon S \setminus U \to G_U$
such that for each $y \in S \setminus U$, the element $x_U(y) := h_U^{-1}(y)$
has the property that $U \cup \{x_U(y)\}$ is a node.
\end{itemize}
\end{defn}

As the name suggests, the structure of an orbit lookup tree makes it easy to find the chosen $G$-orbit representative of a subset of $S$.
\begin{algorithm} \label{algo:find in tree}
Given an orbit lookup tree $T_n$ of depth $n$, for any $k \in \{0,\dots,n\}$ and any sequence $x_1,\dots,x_k$ of distinct elements of $S$, the following recursive algorithm
determines whether $\{x_1,\dots,x_k\}$ is eligible, and if so produces a green node $U$ of $T_n$ and an element $g \in G$ such that $g(U) = \{x_1,\dots,x_k\}$.
\begin{enumerate}
\item
If $k  = 0$, return $U := \emptyset$, $g := 1_G$ and stop.
\item
If $\{x_1,\dots,x_{k-1}\}$ is a node of $T$, let $U'$ be this node and set $g_0 := 1_G$.
Otherwise, apply the algorithm to $x_1,\dots,x_{k-1}$ to obtain a green node $U'$ of $T$ and an element $g_0 \in G$ for which $g_0(U') = \{x_1,\dots,x_{k-1}\}$.
If instead we find that $\{x_1,\dots,x_{k-1}\}$ is not eligible, report that $U$ is not eligible and stop.
\item
Set $y := g_0^{-1}(x_k)$, $U_1 := U' \cup \{x_{U'}(y)\}$, $g_1 := g_0 h_{U'}(y)$.
\item
Set $g_2 := g_{U_1}$ and return $U := g_2^{-1}(U_1)$, $g := g_1 g_2$.
If instead we find that $g_{U_1}$ is undefined, then report that $U$ is not eligible.
\end{enumerate}
\end{algorithm}

\begin{remark} \label{R:modified find in tree}
In Algorithm~\ref{algo:extend tree}, we will use Algorithm~\ref{algo:find in tree}
in a situation where the elements $g_{U_1}$ are not yet computed at depth $k$.
By omitting step (4) and returning $U_1, g_1$, we still obtain a node $U$ of $T_n$ and 
an element $g \in G$ such that $g(U) = \{x_1,\dots,x_k\}$, but without the guarantee that $U$ is green. However, at deeper steps of the recursion we must execute
Algorithm~\ref{algo:find in tree} in full, including step (4).
\end{remark}

The key point is that we can use group retracts to build an orbit lookup tree in an efficient fashion.

\begin{algorithm} \label{algo:extend tree}
Given an orbit lookup tree $T_n$ of depth $n$, the following algorithm extends $T_n$ to an orbit lookup tree $T_{n+1}$ of depth $n+1$.
\begin{enumerate}
\item
For each green node $U$ at depth $n$:
\begin{enumerate}
\item
Choose a sequence $h_1,\dots,h_m$ of generators of $G_U$ by picking random elements.
\item
Construct the Cayley graph $\Gamma_U$ on $S \setminus U$ for the sequence $h_1,\dots,h_m$.
\item
Compute a group retract $(V, g)$ of $\Gamma_U \cup \{\bullet\}$ (with no edges incident to $\bullet$)
and set $h_U := g$.
\item
For each $y \in V$, add to $T_{n+1}$ an uncolored node $U \cup \{y\}$ with parent $U$.
\end{enumerate}
\item
Construct a directed graph with loops $\Gamma$ on the nodes of depth $n+1$ plus the dummy vertex $\bullet$ as follows. For each node $U = [x_1,\dots,x_{n+1}] \in \Gamma$
(optionally in parallel):
\begin{enumerate}
\item
If $U$ is forbidden, then add an edge $\bullet \to U$.
\item
If $U$ is not forbidden, then for $j=1,\dots,n$, apply Algorithm~\ref{algo:find in tree},
as modified in Remark~\ref{R:modified find in tree},
to the sequence 
\[
x_1,\dots,x_{j-1},x_{n+1},x_{j+1}, \dots,x_{n},x_j
\]
to find a node $U_1$ and an element $g_1 \in G$ such that $g_1(U_1) = U$,
and add the edge $U_1 \stackrel{g_1}{\to} U$ to $\Gamma$.
If instead we find that $U$ is not eligible, then add an edge $\bullet \to U$.
\end{enumerate}
\item
Compute a group retract $(V,h)$ of $\Gamma$. Color each vertex in $V$ green
and color each remaining vertex (other than $\bullet$) red.
For each vertex $U$ in an eligible component, set $g_U := h(U)$.
\item
For each green node $U = [x_1,\dots,x_{n+1}]$ at depth $n+1$,
let $G_U$ be the group generated by:
\begin{itemize}
\item
the stabilizer of $x_{n+1}$ in $G_{\{x_1,\dots,x_n\}}$;
\item
for each edge $U_1 \stackrel{g}{\to} U_2$ in $\Gamma$,
the element $g_{U_2}^{-1} g g_{U_1}$.
\end{itemize}
\end{enumerate}
\end{algorithm}
\begin{proof}
The key point here is to confirm that the group computed in step (4), which is evidently contained in the stabilizer of the green node $U = [x_1,\dots,x_{n+1}]$, is actually equal to it. Let $H_U$ be the group computed in (4) and let $G_U$ be the full stabilizer; then the inclusions
\begin{gather*}
 G_{\{x_1,\dots,x_n\}} \cap G_{x_{n+1}} \subseteq H_U \cap G_{\{x_1,\dots,x_n\}} \subseteq G_U \cap G_{\{x_1,\dots,x_n\}} = G_{\{x_1,\dots,x_n\}} \cap G_{x_{n+1}} \\
  G_{\{x_1,\dots,x_n\}} \cap G_{x_{n+1}} \subseteq H_U \cap G_{x_{n+1}}  \subseteq
G_U \cap G_{x_{n+1}} = G_{\{x_1,\dots,x_n\}} \cap G_{x_{n+1}}
\end{gather*}
show that all of these groups coincide.
That is, $x_{n+1}$ has the same stabilizers in $H_U$ and $G_U$, and so the orbit-stabilizer formula implies that the index $[G_U:H_U]$ equals the size of the $G_U$-orbit of $x_{n+1}$ divided by the size of the $H_U$-orbit.
Consequently, it suffices to check that the orbits coincide.

We again identify orbits with the connected components of the Cayley graph. If the $G_U$-orbit of $x_{n+1}$ consists of $x_{n+1}$ itself, there is nothing to check. Otherwise, the orbit also contains $x_j$ for some
$j \in \{1,\dots,n\}$, and the edges arising from the index $j$ in step (2b) guarantee that $x_j$ and $x_{n+1}$ are joined in the Cayley graph.
\end{proof}

\begin{remark}
Algorithm~\ref{algo:find in tree} provides a consistency check for the computation of an orbit lookup tree, as one can spot-verify that a random $k$-element subset is indeed $G$-equivalent to some green node. For our purposes this is sufficient, as we only need a set that covers all orbits, not necessarily a set of orbit representatives.

If one really wants to verify that no two distinct green nodes are $G$-equivalent, it may be easiest to do this using some \emph{ad hoc} computable invariants of the $G$-action.
Alternatively, if no subsets are forbidden, we may verify the orbit-stabilizer formula: the sum of $[G:G_U]$ over green nodes $U$ at depth $n$ should equal $\binom{|S|}{n}$.
\end{remark}

\begin{remark}
We have made no systematic effort to optimize Algorithm~\ref{algo:extend tree} or even to give a careful costing. 
In step (2b) of Algorithm~\ref{algo:extend tree}, a further optimization is possible: instead of taking all $j \in \{1,\dots,n\}$, it suffices to take a set of orbit representatives for the action of $G_{\{x_1,\dots,x_n\}}$ on $\{1,\dots,n\}$.
Our initial experiments were inconclusive as to whether this yielded a meaningful speedup in practice, so we did not pursue it.
\end{remark}

\begin{remark} \label{R:linear case}
In some applications, the set $S$ carries the structure of a $k$-vector space for some finite field $k$, $G$ acts $k$-linearly on $S$,
and one is interested in the action of $G$ on subspaces rather than subsets. One can treat this situation
by considering orbits of linearly independent subsets, but in practice it would be more efficient to adapt the algorithms
to the linear setting. 
As we will not need this here, we omit the details.
\end{remark}

\end{document}